\documentclass[11pt]{article}

\usepackage{fullpage}
\usepackage{amsmath, amsthm, amsfonts, amssymb, amstext, mathrsfs, enumerate}
\usepackage{graphicx, ragged2e, lscape, framed, xcolor}
\usepackage{subfiles}

\theoremstyle{plain}
\newtheorem{theorem}{Theorem}[section]

\newtheorem{problem}[theorem]{Problem}
\newtheorem{claim}{Claim}[section]

\numberwithin{equation}{section}
\allowdisplaybreaks

\newcommand{\affl}[3]{\noindent #1, Email: {\tt #2}\\ \textsc{#3}\\[1.5pt]}

\usepackage[pagebackref]{hyperref}
\hypersetup{
	colorlinks=true,
    urlcolor=purple,
	linkcolor=purple,
    citecolor=purple,
}

\title{\textbf{A somewhat sure note on an un-Schur problem}}
\author{Swaroop Hegde \and Hitesh Kumar \and Pratibha}
\date{}

\newcommand{\Z}{\mathbb Z}

\begin{document}
\maketitle
\begin{abstract} 
Parczyk and Spiegel initiated the study of an anti-Ramsey multiplicity variant of Schur's theorem and proved that the maximum fraction of Schur triples that can be rainbow in a $3$-coloring of $\{ 1, \dots ,n \}$ is bounded asymptotically between $0.4$ and $0.66364$. Furthermore, they conjectured that their lower bound is optimal. We disprove this conjecture and prove new bounds. In particular, we show that the maximum fraction of rainbow Schur triples that can be rainbow in a $3$-coloring of $\{ 1, \dots ,n \}$ lies between $9/22$ and $8/15$ asymptotically. Moreover, we study the problem in the general $k$-color setting and establish new non-trivial bounds. 
\end{abstract}



\section{Introduction}

Schur's Theorem is one of the foundational results in Ramsey theory. It states that in any coloring of $[n]:=\{ 1, \ldots, n\}$ using $k$ colors, there exists a monochromatic solution to the equation $x+y=z$, if $n$ is sufficiently large as a function of $k$. We call a triple $(x,y,z)$ of positive integers a \emph{Schur triple} if it satisfies the Schur equation $x+y=z$. 

Frankl, Graham and R\"{o}dl \cite{FGR_1988} proved a supersaturation result showing that in any $k$-coloring of $[n]$, asymptotically a positive fraction of all Schur triples are monochromatic. The problem of determining the asymptotic minimum in the two-color case was posed by Graham, R\"{o}dl, and Ruci\'{n}ski \cite{Graham_Rodl_Rucinski_1996} and was subsequently popularized by Graham as a $\$100$ problem (cf. \cite{Robertson_Zeilberger_1998}). This problem is now solved, and it is known that the exact minimum fraction is $2/11$ asymptotically due to the works of Robertson--Zeilberger \cite{Robertson_Zeilberger_1998}, Schoen \cite{Schoen_1999}, and Datskovsky \cite{Datskovsky_2003}. Schoen \cite{Schoen_1999} further obtained a characterization of extremal colorings. Similar problems have been studied when $[n]$ is colored using $k\ge 3$ colors (see \cite{DeLoera_Ventura_Wang_Wesley_2025, Mao_Robertson_Wang_Yang_Yang_2026, Thanatipanonda_2009}), but determining the exact asymptotic fraction remains an open problem even in the case $k=3$.

In this paper, we are interested in the anti-Ramsey problem concerning rainbow solutions to the Schur equation. A Schur triple $(x,y,z)$ is said to be \emph{rainbow} if $x, y$ and $z$ have pairwise distinct colors. This problem is interesting only when we color $[n]$ with at least 3 colors. In the 3-color case, Alekseev and Savchev \cite{alekseev1987problem} proved that if each color class in a $3$-coloring of $[3n]$ has the same cardinality, then there must exist a rainbow Schur triple. Sch\"{o}nheim \cite{Schonheim_1990} improved this by showing that it is sufficient to require each color class to have cardinality greater than $n/4$, and proved that this is optimal. 

The problem of counting the number of rainbow Schur triples was only recently posed by Parczyk and Spiegel \cite{Parczyk_Spiegel_2026}. Given $n, k\in \mathbb{N}$ and a $k$-coloring $c$ of $[n]$, let $R_n(c)$ denote the set of ordered pairs $(x,y)$ such that $(x,y,x+y)$ is a rainbow Schur triple with respect to the coloring $c$. Precisely,
\begin{equation}\label{eq:Schur_pairs}
R_n(c):=\{(x,y)\in [n]^2: x+y\le n \text{ and }c(x), c(y), c(x+y) \text{ are distinct} \}.
\end{equation}
 
We count $(x,y)$ and $(y,x)$ separately if $x\ne y$. Note that the total number of ordered pairs $(x,y)\in [n]^2$ that lead to a Schur triple is $n\choose 2$. Our interest is in the parameter
\[ \Lambda_{n,k}:=\frac{1}{{n\choose 2}}\max_{c:[n]\to [k]} |R_n(c)|.\]
In words, $\Lambda_{n,k}$ is the maximum fraction of Schur triples that can be rainbow in $k$-colorings of $[n]$. 

Parczyk and Spiegel \cite{Parczyk_Spiegel_2026} proved that the maximum fraction of rainbow Schur triples, over all $3$-colorings of $[n]$, is bounded asymptotically between $0.4$ and $0.66364$, i.e.,
 \[ 0.4\le \liminf_{n\to \infty} \Lambda_{n,3}\le  \limsup_{n\to \infty} \Lambda_{n,3} \le 0.66364.\]
Furthermore, they conjectured that their lower bound is optimal and that their construction is unique.

Our first main result is the following improved lower and upper bounds in the $3$-color case, disproving the conjecture of Parczyk and Spiegel.

\begin{theorem}\label{thm:lower_upper_bound_3_color} We have  
 \[ 0.40909< \frac{9}{22}\le \liminf_{n\to \infty} \Lambda_{n,3}\le  \limsup_{n\to \infty} \Lambda_{n,3} \le \frac{8}{15}< 0.53334.\]
\end{theorem}

Our proof of the upper bound $8/15$ (Theorem \ref{thm:upper_bound_3_color}) is very short. In \cite{Parczyk_Spiegel_2026}, the well-known connection between Schur triples and triangles in complete graphs is used, together with a rainbow triangle counting result of Balogh et al.~\cite{Balogh_Hu_Lidicky_Pfender_Volec_Young_2017}. The main technicality in their argument involves dealing with a weighted count of rainbow Schur triples, since different Schur triples correspond to different numbers of triangles in their setup. In our proof, we overcome this issue and simplify the argument by working in an edge-coloring of a bigger complete graph defined on an appropriate cyclic group, where every rainbow Schur triple corresponds to the same number of rainbow triangles. We prove the lower bound $9/22$ (Theorem \ref{thm:lower_bound_3_color}) by a careful modification of the construction of Parczyk and Spiegel \cite{Parczyk_Spiegel_2026}.

Our investigations suggest that the optimal asymptotic fraction may be $9/22$. We pose it as a problem.

\begin{problem} Is it true that 
\[ \lim_{n\to \infty} \Lambda_{n,3} = \frac{9}{22}?\]
\end{problem}

A more detailed computational work improves the upper bound to $0.48$, but we choose not to discuss it, as this is still far from the conjectured optimal $9/22$, and to keep our presentation short and clean.

If one considers $k$-colorings of $[n]$ where $k$ is large, then one naturally expects there to be many rainbow Schur triples. We formally prove upper and lower bounds for general $k$. To the best of our knowledge, these are the first non-trivial bounds for $k\geq 4$. 

\begin{theorem}\label{thm:general_lower_upper_bound} For a fixed $k\ge 3$, we have 
\[ 1-\frac{6}{2k+3}\le \frac{(k-2)(k+1)}{(k-1)(k+3)}\le \liminf_{n\to \infty} \Lambda_{n,k}\le \limsup_{n\to \infty} \Lambda_{n,k}\le 1-\frac{1}{k}.\]
\end{theorem}

Note that independently coloring integers in $[n]$ using colors chosen uniformly at random from $[k]$ gives probability $\frac{(k-1)(k-2)}{k^2}$ for a Schur triple to be rainbow, and so immediately one gets 
\[\liminf_{n\to \infty} \Lambda_{n,k} \ge \frac{(k-1)(k-2)}{k^2}.\]
We prove the better lower bound given in Theorem \ref{thm:general_lower_upper_bound} using a slightly non-trivial random coloring of $[n]$ (Theorem \ref{thm:general_lower_bound}). The upper bound is essentially an application of the Cauchy-Schwarz inequality (Theorem \ref{thm:general_upper_bound}). It is clear from Theorem \ref{thm:general_lower_upper_bound} that as $k\to \infty$, the asymptotic maximum fraction of rainbow Schur triples in $k$-colorings of $[n]$ approaches $1$. It is fair to say that for this problem, the case $k=3$ is the most interesting one, and the bounds in Theorem \ref{thm:lower_upper_bound_3_color} are better than those coming from Theorem \ref{thm:general_lower_upper_bound}. We think that the bounds in Theorem \ref{thm:general_lower_upper_bound} are likely not optimal. It would be interesting to close the gap. The $k=4$ case is worth further investigation. 

\section{The $3$-color case}

\subsection{Upper bound when $k=3$}

We now present an argument which improves the upper bound for $\Lambda_{n,3}$ over those in \cite{Parczyk_Spiegel_2026} and Theorem \ref{thm:general_upper_bound}. In \cite{Parczyk_Spiegel_2026} they use the classical trick of converting a coloring of \([n]\) to an edge-coloring of \(K_{n+1}\) and then use the following result of Balogh, Hu, Lidick\'y, Pfender, Volec and
Young \cite{Balogh_Hu_Lidicky_Pfender_Volec_Young_2017} concerning rainbow triangles in edge-colored complete graphs. 

\begin{theorem}[Balogh--Hu--Lidick\'y--Pfender--Volec--Young \cite{Balogh_Hu_Lidicky_Pfender_Volec_Young_2017}]\label{thm:raibow_triangles}
The maximum number $F(m)$ of rainbow triangles in a
three-edge-coloring of $K_m$ satisfies
\[
 F(m)=\left(\frac1{15}+o(1)\right)m^3.
\]
Equivalently, the maximum rainbow-triangle density tends to $2/5$.
\end{theorem}

The main advantage of our proof is that we transfer a coloring of $[n]$ to an edge-coloring complete graph $K_{2n+1}$ in such a way that all rainbow Schur triples correspond to the same number of rainbow triangles, unlike in \cite{Parczyk_Spiegel_2026}. So a direct application of the rainbow triangle bound gives the desired result instead of a weighted count, which lets us bypass the reweighing and optimization steps from \cite{Parczyk_Spiegel_2026}.

\begin{theorem}\label{thm:upper_bound_3_color} We have 
\[\limsup_{n\to \infty} \Lambda_{n,3} \le \frac{8}{15}.\]
\end{theorem}

\begin{proof}
Consider a $k$-coloring $c:[n]\to[k]$. Put $N=2n+1$ and identify
the vertices of $K_N$ with the cyclic group $\Z/N\Z$.  For
$a\ne0$, let
\[
 \|a\|_N:=\min\{r,N-r\}\in[n],
\]
where $r \equiv a \pmod N$ is the least positive representative of $a$. Color the edge
$\{u,v\}$ of $K_N$ by
\(c(\|v-u\|_N).\) This is well defined because $\|v-u\|_N=\|u-v\|_N$. 

For every $t\in\Z/N\Z$ and every ordered pair $(x,y)$ in $R_{n}(c)$, consider the triangle in $K_{N}$ given by the map
\[
 \Phi(t,x,y):=\{t,t+x,t+x+y\},
\]

where the addition is in $\Z/N\Z$. The three pair-wise differences are $x,y,x+y\le n$, so its edge colors are
$c(x),c(y),c(x+y)$ and hence the triangle is rainbow. 

Next, we observe that $\Phi$ is injective. To see this, note that in positive cyclic order the consecutive gaps in $\{t,t+x,t+x+y\}$
are \( x, y\), and \(N-x-y.\) The first two are at most $n$, while
$N-x-y\ge N-n=n+1$. Thus, the third is the unique gap longer than
$n$. The triangle, therefore, recovers $t$ as the vertex
immediately after this long gap, and the next two gaps recover the ordered pair $(x,y)$. So each ordered rainbow Schur triple uniquely corresponds to $N$ distinct rainbow triangles. Therefore,
\[
 N |R_n(c)| \le F(N).
\]
Applying the rainbow triangle bound and then plugging in $N=2n+1$ gives

\[
 |R_n(c)|\le\frac{F(N)}{N}
   =\left(\frac{4}{15}+o(1)\right)n^2 = \left(\frac{8}{15}+o(1)\right)\binom{n}{2},
\]
and the assertion holds. 
\end{proof}

\subsection{Lower bound when $k=3$}

Here, we prove a lower bound for $\Lambda_{n,3}$ using a modification of the coloring in \cite{Parczyk_Spiegel_2026}, which depends both on intervals and the parities of integers in $[n]$. The intervals in the coloring presented here are reminiscent of the 2-coloring of $[n]$ which minimizes the number of monochromatic Schur triples \cite{ Datskovsky_2003, Robertson_Zeilberger_1998, Schoen_1999}.

\begin{theorem}\label{thm:lower_bound_3_color}
There exists a sequence of colorings $c_n:[n]\to [3]$ for which
\[
|R_n(c_n)|
=
\left(\frac{9}{44}+o(1)\right)n^2.
\]
Consequently,
\[
\liminf_{n\to \infty} \Lambda_{n,3}\ge \frac{9}{22}.
\]
\end{theorem}
\begin{proof}
For $n\in \mathbb{N}$, take
\begin{align*}
A& =\left\{i\in[n]: i\text{ is odd and }i\leq\frac{4n}{11}\right\},\\
B& =\left\{i\in[n]: i\text{ is odd and }
\frac{4n}{11}<i\leq\frac{10n}{11}\right\},\\
C&=\left\{i\in[n]: i\text{ is odd and }i>\frac{10n}{11}\right\}.
\end{align*}
Clearly, $A\sqcup B\sqcup C$ is a partition of all odd integers in $[n]$. Define the coloring $c_n$ as follows:
\[
c_n(i)=
\begin{cases}
1 & i\in A;\\
2 & i\in B;\\
1 & i\in C;\\
3 &  i \text{ is even}.
\end{cases}
\]

Next, we count the rainbow Schur triples $(x,y,x+y)$ w.r.t. the coloring $c_n$ according to the parity of the summands $x,y$. By the definition of $c_n$, both $x$ and $y$ cannot be even, or else they receive the same color. So we can partition
\[ R_n(c_n) = R_{\mathrm{oo}} \sqcup R_{\mathrm{oe}},\]
where $R_{oo}$ contains the pairs $(x,y)$ where both $x,y$ are odd, and $R_{\mathrm{oe}}$ contains the pairs where exactly one of $x,y$ is odd. 

\begin{claim}\label{claim:R_oe} We have 
\[|R_{\mathrm{oe}}| = \left(\frac{15}{121}+o(1)\right)n^2.\]
\end{claim}

\begin{proof}
Let $(x,y)\in R_{\mathrm{oe}}$. Then $x+y$ is odd. Assume first that $x$ is odd. Then $y$ is even and $c_n(y)=3$. Also, $c_n(x), c_n(x+y)\in \{1,2\}$ and $c_n(x)\neq c_n(x+y)$. It is then clear that 
\[ (x,x+y)\in A\times B \quad \text{or} \quad (x,x+y)\in B\times C.\]
Thus, the number of such triples is $|A||B| + |B||C|$. By symmetry between $x$ and $y$, it follows that 
\begin{align*}
    |R_{\mathrm{oe}}| & = 2\left(|A||B| + |B||C|\right)\\
    & = 2\left(\frac{2}{11}\frac{3}{11}+\frac{3}{11}\frac{1}{22}\right)n^2+o(n^2)\\
    &=\left(\frac{15}{121}+o(1)\right)n^2.\qedhere
\end{align*}
\end{proof}

\begin{claim}\label{claim:R_oo} We have 
\[|R_{\mathrm{oo}}| = \left(\frac{39}{484}+o(1)\right)n^2.\]
\end{claim}

\begin{proof} Let $(x,y)\in R_{\mathrm{oo}}$. Then $x+y$ is even, and hence $c_n(x+y) = 3$. Also, $c_n(x), c_n(y)\in \{1,2\}$ and $c_n(x)\neq c_n(y)$. Clearly, $x,y$ both cannot belong to $A\cup C$. Moreover, if $x\in B$ and $y\in C$, then
\[
x+y>
\frac{4n}{11}+\frac{10n}{11}>n.
\]
We conclude that $(x,y)$ cannot have one element in $B$ and the other in $C$. It remains to consider pairs $(x,y)$ having one element in $A$ and one in $B$. Define
\[
N_{AB}
=
\left\{(x,y)\in A\times B:x+y\leq n\right\} \quad \text{and}\quad N_{BA}
=
\left\{(x,y)\in B\times A:x+y\leq n\right\}.
\]
It is clear by symmetry that $|N_{AB}| = |N_{BA}|$, and
\[R_{\mathrm{oo}} = N_{AB}\sqcup N_{BA}.\]
So to estimate $|R_{oo}|$, it suffices to estimate $|N_{AB}|$. Fix an odd $y\in B$. The number of odd $x\in A$ satisfying
$x+y\leq n$ is
\[
\frac12
\min\left\{\frac{4n}{11},\,n-y\right\}+O(1).
\]
Using a Riemann-sum approximation, we obtain
\[
|N_{AB}|
=
\frac{n^2}{4}
\int_{4/11}^{10/11}
\min\left\{\frac{4}{11},1-t\right\}\,dt
+O(n).
\]
The point at which the two terms inside the minimum coincide is
$t=7/11$. Hence,
\[
\begin{aligned}
\int_{4/11}^{10/11}
\min\left\{\frac{4}{11},1-t\right\}\,dt
&=
\int_{4/11}^{7/11}\frac{4}{11}\,dt
+
\int_{7/11}^{10/11}(1-t)\,dt\\
&=
\frac{12}{121}+\frac{15}{242}\\
&=
\frac{39}{242}.
\end{aligned}
\]
It follows that
\[
|N_{AB}|
=
\left(\frac{39}{968}+o(1)\right)n^2,
\]
and therefore
\[
|R_{\mathrm{oo}}|
=
2|N_{AB}|
=
\left(\frac{39}{484}+o(1)\right)n^2.\qedhere
\]
\end{proof}
Using Claims \ref{claim:R_oe} and \ref{claim:R_oo}, we conclude that
\[    |R_n(c_n)| = |R_{\mathrm{oe}}|+|R_{\mathrm{oo}}| = \left(\frac{9}{44}+o(1)\right)n^2.\qedhere\]
\end{proof}

\section{The $k$-color case}

\subsection{General upper bound}

Here, we establish the following upper bound for $\Lambda_{n,k}$ as an application of the Cauchy-Schwarz inequality. As remarked earlier, for $k=3$ this upper bound is weaker than the one in Theorem \ref{thm:upper_bound_3_color} but it comes very close to the upper bound of Parczyk and Spiegel \cite{Parczyk_Spiegel_2026}, and the proof does not depend on rainbow triangle counting.

\begin{theorem} \label{thm:general_upper_bound}
For every $k\ge 3$, we have 
\[ \limsup_{n\to \infty} \Lambda_{n,k}\le 1-\frac{1}{k}.\]
\end{theorem}

\begin{proof}
Consider a $k$-coloring $c:[n]\to [k]$ and let $n_i = |c^{-1}(i)|$, i.e., $n_i$ is the size of the $i$-th color class. 

Observe that the map $\displaystyle \Phi:R_n(c)\to {[n]\choose 2}$ given by 
\[ \Phi(x,y) = \{x,x+y\}\]
is injective, since $x+y$ is larger than $x$, and knowing $x$ and $x+y$ determines $(x,y)$. Moreover, the two elements of $\Phi(x,y)$ have different colors. Therefore, 
\begin{align*}
    |R_n(c)|& \le \sum_{i<j}n_in_j\\
    & = \frac{1}{2}\left(n^2 - \sum_{i=1}^k n_i^2\right)\\
    & \le \frac{1}{2}\left(n^2-\frac{\left(\sum_{i=1}^k n_i\right)^2}{k}\right)\quad (\text{Cauchy-Schwarz})\\
    & = \frac{1}{2}\left(1-\frac{1}{k}\right)n^2. 
\end{align*}
Thus,
\[ \Lambda_{n,k} = \frac{1}{{n\choose 2}}\max_c |R_n(c)|\le 1-\frac{1}{k} + o(1),\]
and the assertion holds. 
\end{proof}

\subsection{General lower bound}

Here, we prove the following lower bound for $\Lambda_{n,k}$ using a probabilistic argument. The idea is to color the initial interval with a single color and then randomly color the remaining interval.

\begin{theorem}\label{thm:general_lower_bound} For $k\ge 3$, we have 
\[ \liminf_{n\to \infty} \Lambda_{n,k} \ge  \frac{(k-2)(k+1)}{(k-1)(k+3)}.\]
\end{theorem}

\begin{proof} Let $m\in \mathbb{N}$ be a parameter, which we will choose later, such that $m\le \frac{1}{2}n$. We construct a random coloring $c:[n]\to [k]$ as follows: every integer in $[m]$ receives color $1$, while each integer in $[n]\setminus [m]$ receives a random color from $\{2, \ldots, k\}$ uniformly and independently. We will now compute $\mathbb{E}[|R_n(c)|] $, i.e., the expected number of rainbow Schur triples with respect to the random coloring $c$. To prove the assertion, it suffices to show that 
\[ \liminf_{n\to \infty} \frac{\mathbb{E}[|R_n(c)|]}{{n\choose 2}}\ge \frac{(k-2)(k+1)}{(k-1)(k+3)}. \]

To that end, we classify the ordered pairs $(x,y)$ in $R_n(c)$ according to whether $x,y$ are bigger or smaller than $m$. Note that if both $x$ and $y$ are at most $m$, then both $x$ and $y$ receive the same color, and so this case is not possible. Hence, at least one of $x$ and $y$ is larger than $m$. Thus, we can write 
\[ R_n(c) = A \sqcup B,\]
where 
\[A = \{(x,y)\in R_n(c): \text{exactly one of }x,y\text{ is more than }m\},\]
and 
\[ B = \{(x,y)\in R_n(c): \text{both }x,y\text{ are more than }m\}.\]

By linearity of expectation,
\begin{equation}\label{eq:exp_R_n_c}
   \mathbb{E}[|R_n(c)|] = \mathbb{E}[|A|] + \mathbb{E}[|B|]. 
\end{equation}

We now compute $\mathbb{E}[|A|]$ and $\mathbb{E}[|B|]$. 

\begin{claim}\label{claim:A} We have
\[ \mathbb{E}[|A|] = \frac{k-2}{k-1}\left(2mn - 3m^2 - m\right).\]
\end{claim}

\begin{proof} Consider an ordered pair $(x,y)$ and let $\alpha = \max\{x,y\}>m$. Then $x+y$ is also more than $m$, and the colors of $\alpha$ and $x+y$ are chosen independently out of the $k-1$ colors. Thus, the probability that $(x,y) \in A$ is 
\begin{equation}\label{eq:prob_A}
    \frac{(k-1)(k-2)}{(k-1)^2}=\frac{k-2}{k-1}.
\end{equation}

Now, we wish to count pairs $(x,y)$ satisfying
\[ \min\{x,y\}\le m, \quad  m+1\le \max\{x,y\}, \quad x+y\le n.\]
If we fix $x\leq m$, then this implies $y> m$. If $(x,y) \in R_{n}(c)$ then $y$ must lie between $m+1$ and $n-x$, so that $x+y\leq n$. A symmetric condition holds for the pair $(y,x)$. Thus, the number of such possible pairs is 
\begin{align}\label{eq:count_A}
2\sum_{x=1}^m |\{y: m+1\le y\le n-x\}| 
& = 2\sum_{x=1}^m (n-m-x)\nonumber\\
& = 2mn - 3m^2 - m.
\end{align}
The claim follows by \eqref{eq:prob_A} and \eqref{eq:count_A}.
\end{proof}

\begin{claim}\label{claim:B} We have 
\[ \mathbb{E}[|B|] = \frac{(k-2)(k-3)}{(k-1)^2}\ \left[{n-2m\choose 2} - \left(\left\lfloor\frac{n}{2}\right\rfloor - m\right)\right].\]
\end{claim}

\begin{proof} Consider an ordered pair $(x,y)$ such that both $x$ and $y$ are more than $m$. Then $x,y,x+y$ all receive colors independently out of the $k-1$ colors, and so the probability that $(x,y)\in B$ is 
\begin{equation}\label{eq:prob_B}
     \frac{(k-1)(k-2)(k-3)}{(k-1)^3} = \frac{(k-2)(k-3)}{(k-1)^2}.
\end{equation}

Now, we wish to count pairs $(x,y)$ with $x\ne y$ satisfying 
\[ m+1\le x, \quad m+1\le y, \quad x+y\le n.\]
Using the substitution $u=x-m$ and $v=y-m$, the above conditions become 
\[ 1\le u, \quad 1\le v, \quad u+v\le n-2m.\]
So the total number of pairs (allowing $u=v$ or equivalently $x=y$) is 
\begin{align*}
\sum_{u=1}^{n-2m-1} |\{v: 1\le v\le n-2m-u\}|
& = \sum_{u=1}^{n-2m-1} (n-2m-u)\\
& = {n-2m \choose 2}. 
\end{align*}
The number of pairs of the form $(x,x)$ such that $m+1\le x$ and $2x\le n$ is \[\left(\left\lfloor\frac{n}{2}\right\rfloor - m\right).\] 
And so the total number of such pairs is
\begin{equation}\label{eq:count_B}
    {n-2m\choose 2} - \left(\left\lfloor\frac{n}{2}\right\rfloor - m\right).
\end{equation}
The claim follows by \eqref{eq:prob_B} and \eqref{eq:count_B}.
\end{proof}

Now, using \eqref{eq:exp_R_n_c} and Claims \ref{claim:A} and \ref{claim:B}, we get 
\begin{align*}
\mathbb{E}[|R_n(c)|] 
& = \frac{k-2}{k-1}\left(2mn - 3m^2 - m\right) + \frac{(k-2)(k-3)}{(k-1)^2}\ \left[{n-2m\choose 2} - \left(\left\lfloor\frac{n}{2}\right\rfloor - m\right)\right].
\end{align*}

Assume $\frac{m}{n}\to a$ as $n\to \infty$, where $0\le a\le\frac{1}{2}$. Then 
\begin{align*}
\lim_{n\to \infty}\frac{\mathbb{E}[|R_n(c)|]}{{n\choose 2}}
& = \frac{k-2}{k-1}\left(4a - 6a^2\right) + \frac{(k-2)(k-3)}{(k-1)^2}(1-2a)^2.
\end{align*}
With $k$ fixed, the right-hand side above is maximized over $a\in [0,\frac{1}{2}]$, when $a = \frac{2}{k+3}$. Evaluating the right-hand side at $a=\frac{2}{k+3}$ gives 
\begin{align*}
\lim_{n\to \infty}\frac{\mathbb{E}[|R_n(c)|]}{{n\choose 2}}
& = \frac{(k-2)(k+1)}{(k-1)(k+3)}.
\end{align*}
This completes the proof. 
\end{proof}

\section*{Declaration of AI use}

The authors acknowledge the use of ChatGPT (GPT-5.6, OpenAI; accessed July-August 2026) solely for preliminary brainstorming and the exploration of possible proof strategies. AI tools were not used to draft the manuscript. All formal statements, arguments, and proofs in the manuscript were written and checked by the authors, who take full responsibility for the accuracy and integrity of the article.

\bibliography{citations}
\bibliographystyle{plain}

\vspace{0.4cm}

\affl{Swaroop Hegde}{swaroop.hegde@uga.edu}{Department of Mathematics, University of Georgia, Athens, United States}

\affl{Hitesh Kumar}{hitesh.kumar.math@gmail.com, hitesh\_kumar@sfu.ca}{Department of Mathematics, Simon Fraser University, Burnaby, Canada}

\affl{Pratibha}{pratibha22320@gmail.com}{Bhiwani, India}

\end{document}